\documentclass[reqno, 12pt, reqno]{amsart}

\usepackage{amsfonts, amsthm, amsmath, amssymb}
\usepackage{hyperref}
\hypersetup{colorlinks=false}
\usepackage{comment}

\usepackage[margin=3.4cm]{geometry}

\RequirePackage{mathrsfs}
\let\mathcal\mathscr

\numberwithin{equation}{section}

\newtheorem{theorem}{Theorem}[section]

\newtheorem{lemma}[theorem]{Lemma}

\newtheorem{proposition}[theorem]{Proposition}

\theoremstyle{definition}

\newtheorem{remark}[theorem]{Remark}

\newtheorem*{remark*}{Remark}

\newcommand{\BA}{{\mathbb {A}}}

\newcommand{\BC}{{\mathbb {C}}}

\newcommand{\BL}{{\mathbb {L}}}

\newcommand{\BN}{{\mathbb {N}}}

\newcommand{\BP}{{\mathbb {P}}}
\newcommand{\BQ}{{\mathbb {Q}}}
\newcommand{\BR}{{\mathbb {R}}}
\newcommand{\BS}{{\mathbb {S}}}

\newcommand{\BZ}{{\mathbb {Z}}}
\newcommand{\CA}{{\mathcal {A}}}
\newcommand{\CB}{{\mathcal {B}}}

\newcommand{\CF}{{\mathcal {F}}}
\newcommand{\CG}{{\mathcal {G}}}
\newcommand{\CH}{{\mathcal {H}}}
\newcommand{\CI}{{\mathcal {I}}}
\newcommand{\CJ}{{\mathcal {J}}}
\newcommand{\CK}{{\mathcal {K}}}

\newcommand{\CN}{{\mathcal {N}}}

\newcommand{\CS}{{\mathcal {S}}}
\newcommand{\CT}{{\mathcal {T}}}
\newcommand{\CU}{{\mathcal {U}}}
\newcommand{\CV}{{\mathcal {V}}}

\newcommand{\bb}{{\mathbf{b}}}

\renewcommand{\phi}{\varphi}
\renewcommand{\rho}{\varrho}
\renewcommand{\epsilon}{\varepsilon}

\newcommand{\x}{\mathbf{x}}

\newcommand{\bx}{\boldsymbol{x}}
\newcommand{\by}{\boldsymbol{y}}
\newcommand{\bz}{\boldsymbol{z}}
\newcommand{\bt}{\boldsymbol{t}}
\newcommand{\bs}{\boldsymbol{s}}

\newcommand{\bc}{\boldsymbol{c}}

\newcommand{\bu}{\boldsymbol{u}}

\newcommand{\blambda}{\boldsymbol{\lambda}}

\newcommand{\bsigma}{\boldsymbol{\sigma}}

\newcommand{\bbeta}{\boldsymbol{\beta}}

\newcommand{\e}{\textup{e}}

\title[Ternary quadratic forms II]{Quantitative strong approximation for ternary quadratic forms II}
\author{Zhizhong Huang}

\address{Institute of Mathematics, Academy of Mathematics and Systems Science, Chinese Academy of Sciences, Beijing, 100190, China}
\email{zhizhong.huang@yahoo.com}
\date{November 2024}

\begin{document}
	\begin{abstract}
		Let $F$ be a non-degenerate integral ternary  quadratic form and let $m_0\in\BZ_{\neq 0}$. We study growth of rational points on the affine quadric $(F=m_0)$ and show that they are equidistributed in the adelic space off a finite place. This is closely related to Linnik's problem. Our approach is based on the $\delta$-variant of the Hardy--Littlewood circle method developed by Heath-Brown.
	\end{abstract}
	\maketitle
	\tableofcontents
	
	\section{Introduction}
	Let $F$ be an integral non-degenerate  quadratic form in $d$-variables and let $m_0\in\BZ_{\neq 0}$, which will be fixed throughout. For a growing parameter $N\in\BN$, we denote by $V_N$ the affine quadric $(F=m_0N)\subset \BA^d$ over $\BQ$ with the integral model $\CV_N$ defined by the same equation. Assume that $V_1(\BR)\neq\varnothing$.
	Consider the projection map
	\begin{align*}
		\pi_N:V_N(\BR) &\longrightarrow V_1(\BR)\\
		\bx&\longmapsto \frac{\bx}{\sqrt{N}}.
	\end{align*}
	Linnik's problem \cite{Linnik} is concerned with the distribution of the integral points $\CV_N(\BZ)$ under the map $\pi_N$ on $V_1(\BR)$.
	
		Let $w:\BR^d\to\BR$ be a weight function (infinitely differentiable and compactly supported). Suppose that $\operatorname{Supp}(w)\cap V_{1}(\BR)$ is a compact region of $V_1(\BR)$ with non-empty interior. Let $L\in\BN$ be fixed. To each $N\to\infty$, we associate a residue $\blambda_N\in \CV_N(\BZ/L\BZ)$, viewed as an element of $\BZ^d$ of size $\ll L$ satisfying $F(\blambda_N)\equiv m_0N\bmod L$. To these data we associate the weighted  counting function:
	\begin{equation}\label{eq:countingfun}
		\Gamma_{w}(N):=\sum_{\substack{\bx\in\BZ^d,F(\bx)=m_0N\\ \bx\equiv \blambda_N\bmod L}}w\left(\frac{\bx}{\sqrt{N}}\right).
	\end{equation}
	Observe that this is a finite sum because it only counts integer vectors $\bx\in\sqrt{N}\operatorname{Supp}(w)$.
	
	Past applications of the Hardy--Littlewood circle method had been successful only when $d\geqslant 4$ (the breakthrough for $d=4$ goes back to Kloosterman \cite{Kloosterman} almost a century ago). So we shall fixate on the \emph{ternary} case $d=3$, where was long considered as beyond reach of the circle method. The case where $F$ is positive definite (and with a different type of congruence conditions) has been extensively studied by many authors \cite{Duke,Duke-SchP, Golubeva-Fomenko, Iwaniec}, and it is shown that $\CV_N(\BZ)$, as $N\to\infty$ (subjected to appropriate conditions), become equidistributed on $V_1(\BR)$ via the map $\pi_N$. When $F$ is assumed to be anisotropic over $\BQ$ and $N$ has only restricted prime divisors, see also \cite{Benoist-Oh}.
	
	The goal of this article is to push the circle method further (for the first time as far as the author is aware) 
	to address Linnik's type ordering problem for ternary quadratic forms. 

	\subsection{Main results on equidistribution of $\BZ[\frac{1}{p_0}]$-points}
Let us fix throughout a prime $p_0$, and let $V_1(\mathbf{A}_{\BQ}^{p_0})$ be the \emph{adelic space off $p_0$} of $V_1$, which is by definition the image of of adelic space $V_1(\mathbf{A}_{\BQ})$ (the restricted product over all $V(\BQ_p)$ against integral models of $V_1$) under the projection without the $p_0$-component. Let $V_1(\mathbf{A}_{\BQ})^{\operatorname{Br}}$ be the closed subset of $V_1(\mathbf{A}_{\BQ})$ consisting of elements that are orthogonal to the Brauer group of $V_1$.
Thanks to work of Colliot-Thélène--Xu \cite[\S5.6, \S5.8]{CT-XuCompositio} \cite[Proposition 4.5]{CT-Xu}, provided that  the quadratic form $F$ is isotropic over $\BQ_{p_0}$, the set $V(\BQ)$ is dense inside the image of $V_1(\mathbf{A}_{\BQ})^{\operatorname{Br}}$ in $V_1(\mathbf{A}_{\BQ}^{p_0})$. We refer to \cite[\S1]{CT-XuCompositio} and \cite[\S1]{CT-Xu} for more details about terminology, based on which we say that strong approximation with Brauer--Manin obstruction off the prime $p_0$  holds for the affine quadric $V_1$.  

 In our current investigation, motivated by this qualitative aspect, we consider the adelic subspace $$\CV_1\left(\widehat{\BZ}\left[\frac{1}{p_0}\right]\right):=V_1(\BR)\times\prod_{\substack{p<\infty, p\neq p_0}}\CV_1(\BZ_p)\subset  V_1(\mathbf{A}_{\BQ}^{p_0}),$$ which we tacitly assume to be non-empty.  Let us fix $\blambda\in \CV_1(\BZ/L\BZ)$, viewed as a vector in $\blambda\in\BZ^3$ of size $\ll L$ such that $F(\blambda)\equiv m_0\bmod L$. As $(L,\blambda)$ captures non-archimedean local conditions outside of $p_0$, we assume that $p_0\nmid L$.  The collection $\{w,\CV_1,(L,\blambda)\}$ specifies a (non-empty) open-closed subset of $\CV_1\left(\widehat{\BZ}\left[\frac{1}{p_0}\right]\right)$. 

To quantify the strong approximation property off $p_0$ for $V_1$, we are interested in counting rational points on $V_1$ satisfying the local conditions prescribed by the collection $\{w,\CV_1,(L,\blambda)\}$. If for every prime $p$, we equip the affine $3$-space $\BA^3_{\BQ}$ with the $p$-adic height $$\|(x_1,x_2,x_3)\|_{p}:=\max_{1\leqslant i\leqslant 3}|x_i|_{p},$$ 
then the growth of $V_1(\BQ)$ with these prescribed local conditions may be quantified by the counting function $$\#\left\{\bx\in V_1\left(\BQ\right)\cap \CV_1\left(\widehat{\BZ}\left[\frac{1}{p_0}\right]\right)\left|\begin{aligned}
	&\|\bx\|_{p_0}\leqslant p_0^h\\ &\bx\in \operatorname{Supp}(w)\text{ and }\forall p\mid L,\|\bx-\blambda\|_{p}\leqslant p^{-\operatorname{ord}_{p}(L)}
\end{aligned}\right.\right\}.$$ Here $h\in\BN$ is a growing parameter. 
We therefore let \begin{equation}\label{eq:Np0}
 	N=p_0^{2h}, \quad \blambda_N:=p_0^h\blambda,
 \end{equation} so that the counting function  $\Gamma_{w}(N)$ \eqref{eq:countingfun} ``mollifies'' the problem above. We would like to stress that the quadratic form $F$ need not be positive definitive or $\BQ$-anisotropic. 
	
	We define the \emph{weighted singular integral} by \begin{equation}\label{eq:singint}
		\widetilde{\CI}(w):=\int_{\BR^3\times\BR}w(\bt)\e(\theta (F(\bt)-m_0))\operatorname{d}\bt\operatorname{d}\theta. 
	\end{equation} It computes the (weighted) real density of the affine quadric $V_1$ restricted to the support of $w$ (cf. e.g. \cite[Theorem 3]{H-Bdelta}).  
	
	We denote by $\psi_0$ the real character $\left(\frac{-m_0\Delta_F}{\cdot}\right)$, being principal or not. Let us define, for every prime $p$, 
	$$\sigma_{p}(\CV_1;(L,\blambda)):=\lim_{k\to\infty}\frac{\#\{\bx\in(\BZ/p^k\BZ)^3:F(\bx)\equiv m_0\bmod p^k,\bx\equiv \blambda\bmod p^{\operatorname{ord}_p(L)}\}}{p^{2k}}.$$
	For $p_0$ (we call that $p_0\nmid L$), we let
	$$\sigma_{p_0}(\CV_0):=\lim_{k\to\infty}\frac{\#\{\bx\in(\BZ/p_0^k\BZ)^3:F(\bx)\equiv 0\bmod p_0^k\}}{p_0^{2k}},$$ which is the $p_0$-adic density of the affine cone $$\CV_0:=(F=0)\subset\BA^3_{\BZ},$$ and is positive if and only if the quadratic form $F$ is isotropic over $\BQ_{p_0}$.
	 We define the ``modified singular series''  $\widetilde{\mathfrak{S}}_F(p_0;L,\blambda)$  as follows, the convergence of which will be established in course of proof.
	 \begin{footnotesize}
	\begin{equation}\label{eq:singser}
	\widetilde{\mathfrak{S}}_F(p_0;L,\blambda):=	\begin{cases}
		\displaystyle	\left(1-\frac{\psi_0(p_0)}{p_0}\right)\sigma_{p_0}(\CV_0)\times\prod_{\substack{p<\infty\\ p\neq p_0}}\left(1-\frac{\psi_0(p)}{p}\right)\sigma_{p}(\CV_1;(L,\blambda)) &\text{ if } -m\Delta_F\neq\square;\\
		\displaystyle	\left(1-\frac{1}{p_0}\right)\sigma_{p_0}(\CV_0)\times\prod_{\substack{p<\infty\\ p\neq p_0}}\left(1-\frac{1}{p}\right)\sigma_{p}(\CV_1;(L,\blambda)) &\text{ if } -m\Delta_F=\square.
		\end{cases}
	\end{equation} 
	 \end{footnotesize}
	\begin{theorem}\label{thm:mainsq}
		Assume that \begin{equation}\label{eq:m0deltasq}
			-m_0\Delta_F=\square.
		\end{equation}
	Then there exists $a_h\in\BC$ depending on $\{w,\CV_1,(L,\blambda)\}$ and $h$ such that
	$$\Gamma_w(N)=\widetilde{\CI}(w)\widetilde{\mathfrak{S}}_F(p_0;L,\blambda)\sqrt{N}\log \sqrt{N}+a_h\sqrt{N}+O\left(\frac{\sqrt{N}}{(\log N)^{A}}\right),$$ where $A>0$ is an effectively computable numerical constant. Moreover $a_h=O(1)$.
	\end{theorem}
The condition \eqref{eq:m0deltasq} ensures the triviality of the Brauer group of $V_1$, whence $V_1(\mathbf{A}_{\BQ})^{\operatorname{Br}}=V_1(\mathbf{A}_{\BQ})$. Theorem \ref{thm:mainsq}  provides a quantitative aspect of how $\CV_1(\BZ[\frac{1}{p_0}])$ equidistributes in $\CV_1\left(\widehat{\BZ}\left[\frac{1}{p_0}\right]\right)$, in particular in every factor $\CV_1(\BZ_{p}),p\neq p_0$ as well as $V_1(\BR)$. Note that they accumulate under the $p_0$-adic topology around the boundary divisor $(F=0)$ of the naive compactification of $V_1$ in $\BP^3$. This is in accordance with the Chambert-Loir--Tschinkel heuristic (cf. \cite[Appendix]{Huang}).

\begin{theorem}\label{thm:mainnotsq}
	Assume that 	\begin{equation}\label{eq:notsq}
		-m_0\Delta_F \neq\square.
	\end{equation} 		
Then there exists $b_h\in\BC$ depending on $\{w,\CV_1,(L,\blambda)\}$ and $h$ such that 
$$\Gamma_w(N)=\sqrt{N}\left(\widetilde{\CI}(w)\widetilde{\mathfrak{S}}_F(p_0;L,\blambda)+b_h\right)+O\left(\frac{\sqrt{N}}{(\log N)^{A}}\right),$$ with $A>0$ as before and $b_h=O(1)$.
\end{theorem}
	
	Under \eqref{eq:notsq}, results of similar sort are also obtained by Benoist--Oh \cite[Corollaries 1.11 \& 1.12]{Benoist-Oh} (with an inexplicit leading constant) using mixing properties of Lie groups. \footnote{Note that in \cite{Benoist-Oh} it is assumed that the quadratic form $F$ is anisotropic over $\BQ$. In  our setting this condition turns out to be equivalent to \eqref{eq:notsq} if $V_1(\BQ)\neq\varnothing$.}
	
	\subsection{Methods of proof and comments}
Our proof of Theorems \ref{thm:mainsq} \ref{thm:mainnotsq} is based upon the $\delta$-variant of the circle method developed by Heath-Brown \cite{H-Bdelta} (and goes back to \cite{D-F-I}), and is underpinned by the blueprint laid down in \cite{Huang}. Having disposed of the ingredients developed in \cite{Huang} dealing with various sums of Salié sums, the main technical challenge of the current article is to tackle the oscillatory integrals, especially their partial derivatives. Due to the inhomogeneous form of the equation defining $V_1$, we need to go beyond the estimates furnished by \cite[Lemma 22]{H-Bdelta}, and have to contend with the appearance of oscillatory factors while differentiating the oscillatory integrals in the treatment of both  exceptional  and especially ordinary Poisson variables. We refer to Remarks \ref{rmk:exceptionalc} and \ref{rmk:ordc} respectively for more explanations.
	
	Under the assumption \eqref{eq:notsq} the affine quadric $V_1$ carries a non-trivial Brauer group that can evaluate non-constantly (cf. e.g. \cite[Lemma 4.3]{CT-Xu}).
	 For instance, it can happen that these local conditions specify a closed subset of $V_1(\mathbf{A}_{\BQ})$ which does not intersect $V_1(\mathbf{A}_{\BQ})^{\operatorname{Br}}$. In that case $\Gamma_{w}(p_0^{2h})= 0$ for any $h$ and we must have $b_h=-\widetilde{\CI}(w)\widetilde{\mathfrak{S}}_F(p_0;L,\blambda)$. However, apart from this particular occasion, due to its complicated (though explicit) form  and  its subtle dependence on how $h\to\infty$, we are not quite able to compute it directly from its expression.\footnote{Prof. Dasheng Wei has kindly informed us that he is able to show that $\Gamma_w(N)\gg \sqrt{N}$ if the the local conditions are unobstructed, using the universal torsor method. From this we can infer that in this case $|b_h|<\widetilde{\CI}(w)\widetilde{\mathfrak{S}}_F(p_0;L,\blambda)$.} Likewise we do not have a good understanding of $a_h$, though $a_h,b_h$ resemble fairly similar to each other, and both arise in roughly the same way as the secondary term in \cite[Theorem 1.2]{Huang}. 
	 
	 In the case where $F$ is positive definitive we can say a bit more about $b_h$, on appealing to work of Duke--Schulze-Pillot \cite{Duke-SchP}. Let us assume for simplicity that $L=1$, i.e. there is no congruence condition. The product of local densities of $\CV_1$, according to Siegel's mass formula (cf. \cite[Theorem 2 (i)]{Duke-SchP}), equals a weighted sum of representation masses of forms of the same genera as $F$. And \cite[Theorem 5]{Duke-SchP} shows that the leading constant for the number of representations of by $F$ equals a weighted sum of  representation masses of forms of the same \emph{spinor} genera as $F$. So $b_h$ measures the difference of these two weighted sums, and according to \cite[Theorem 2 (ii)]{Duke-SchP}, it equals a product of local densities, defined in terms of spinor norms, which can be reinterpreted via evaluation of the Brauer group by work of Colliot-Thélène--Xu \cite{CT-XuCompositio}. From this heuristic,  we expect that in general the term $b_h$ should encode the effect of Brauer--Manin obstruction regarding our imposed local conditions.
	
Keeping our motivation in mind, we content ourselves in the current article with the assumption \eqref{eq:Np0}. In view of Linnik's problem, we do not seek to explore the weakest assumption for $N$ under which our arguments apply unconditionally. Presumably we may flexibly vary  the square part of $N$ and allow  the largest square-free prime divisor of $N$ to grow ``slowly'' perhaps like a power of $\log$. For more general $N$ the ``Grand Riemann Hypothesis'' might be needed. 
\subsection{Plan of the article}
In \S\ref{se:delta}, without any restriction on $N$, we briefly recall how the $\delta$-method \cite{H-Bdelta} is executed for the counting function \eqref{eq:countingfun}, and we gather together the arithmetic ingredients developed in \cite{Huang}.  In \S\ref{se:counting}, depending on the nature of the Poisson variables $\bc$, we analyse in detail their contributions assuming \eqref{eq:Np0}. We complete the proof of Theorems \ref{thm:mainsq} \ref{thm:mainnotsq} in \S\ref{se:completeproof}.

\subsection{Notation}
The letter $p$ is always reserved for a prime number. We write $\e_n(\cdot)=\exp\left(\frac{2\pi i}{n}\cdot\right)$. 
We recall some notation inherited from \cite{PART1,Huang}. Let $\Delta_F$ denote the determinant of $F$ and let $\Omega:=2L\Delta_F$. For any $\kappa>0$, let $\Upsilon_{\kappa}$ be the arithmetic function 
	$$\Upsilon_{\kappa}(n):=\prod_{p\mid n}\left(1-p^{-\kappa}\right)^{-1}.$$
For any $l,q\in\BN$, we write $$q_{l}:=\prod_{p\mid l}p^{\operatorname{ord}_{p}(q)}.$$

	\section{Activating the $\delta$-method}\label{se:delta}
	Using  the $\delta$-method à la Heath-Brown \cite[Theorem 1]{H-Bdelta} and executing Poisson summation, we get, based on well-known manipulation as in e.g. \cite[Proposition 2.5]{PART1}
	\begin{equation}\label{eq:Poisson}
			\begin{split}
			\Gamma_w(N)&=\sum_{\substack{\by\in\BZ^3:L^2\mid F(L\by+\blambda_N)-m_0N}} w\left(\frac{L\by+\blambda_N}{\sqrt{N}}\right)\delta\left(\frac{F(L\by+\blambda_N)-m_0N}{L^2}\right)\\ &=\frac{C_Q}{Q^2}\sum_{q\in\BN}\sum_{\bc\in\BZ^3}\frac{\widetilde{S}_{q}(\bc)\widetilde{I}_{q}(\bc)}{(qL)^{3}}
		\end{split} 
	\end{equation}
for any $Q>1$,
	 where  $$C_Q=1+O_N(Q^{-N}),$$	and \begin{equation}\label{eq:Sqc}
		\widetilde{S}_{q}(\bc):=\sum_{\substack{a\bmod q\\(a,q)=1}}\sum_{\substack{\bsigma\in(\BZ/qL\BZ)^3\\  L^2\mid F(L\bsigma+\blambda_N)-m_0N}}\textup{e}_{qL}\left(a\left(\frac{F(L\boldsymbol{\sigma}+\blambda_N)-m_0N}{L}\right)+\bc\cdot\bsigma\right),
	\end{equation}
	\begin{equation}\label{eq:Iqc}
			\widetilde{I}_{q}(\bc):=\int_{\BR^3}w\left(\frac{L\bs+\blambda_N}{\sqrt{N}}\right)h\left(\frac{q}{Q},\frac{F(L\bs+\blambda_N)-m_0N}{(LQ)^2}\right)\e_{qL}\left(-\bc\cdot\bs\right)\operatorname{d}\bs 
		\end{equation}
 Here $h:(0,\infty)\times\BR\to\BR$ is a ``nice'' infinitely differentiable function of class $\CH$ defined in \cite[\S4]{H-Bdelta}.
	We choose \begin{equation}\label{eq:Q}
		Q:=\frac{\sqrt{N}}{L}.
	\end{equation}
	The quantities above all depend on $N$.
	 \subsection{Quadratic exponential sums}
	For every factorisation $$q=q_1q_2,$$ with $\gcd(q_1,q_2\Omega)=1$, by the Chinese remainder theorem, we can decompose $\widetilde{S}_{q}(\bc)$ as in \cite[Lemma 4.2]{PART1}: $$\widetilde{S}_{q}(\bc)=\widetilde{S}_{q}^{(1)}(\bc)\widetilde{S}_{q}^{(2)}(\bc),$$
where  \begin{equation}\label{eq:S1}
	\widetilde{S}_{q}^{(1)}(\bc):=\sum_{\bsigma_1\bmod q_1}\sum_{\substack{a_1\bmod q_1\\(a_1,q_1)=1}}\e_{q_1}\left(a_1\left(F(q_2L^2\bsigma_1+\blambda_N)-m_0N\right)+\bc\cdot\bsigma_1\right),
\end{equation} and \begin{equation}\label{eq:S2}
\widetilde{S}_{q}^{(2)}(\bc):=\sum_{\substack{\bsigma_2\bmod q_2L\\ L^2\mid F(Lq_1\bsigma_2+\blambda_N)-m_0N}}\sum_{\substack{a_2\bmod q_2\\(a_2,q_2)=1}}\e_{q_2L}\left(a_2\left(\frac{F(Lq_1\bsigma_2+\blambda_N)-m_0N}{L}\right)+\bc\cdot\bsigma_2\right).
\end{equation} 

	The following gives precise evaluation of $\widetilde{S}_{q}^{(1)}(\bc)$. 
	\begin{lemma}[cf. \cite{Huang} Proposition 2.1]\label{le:S1}
		We have, whenever $(q_1,m_0N)=1$, 
	$$\widetilde{S}_{q}^{(1)}(\bc)=\e_{q_1}\left(-\overline{q_2L^2}\blambda_N\cdot \bc\right)q_1^2\left(\frac{-m_0N\Delta_F}{q_1}\right)\sum_{\substack{u\bmod q_1\\ G_{\bc,q_2}(u)\equiv 0\bmod q_1}}\e_{q_1}(u),$$
	where for $l\in\BN,\bc\in\BZ^3$, we define the quadratic polynomial
\begin{equation}\label{eq:polyG}
		G_{\bc,l}(T):=(\Delta_F T)^2-(lL^2)^2  m_0 N \Delta_FF^*(\bc),
\end{equation} and $F^*$ denotes the dual form of $F$.
	\end{lemma}

The following is a uniform estimate for $\widetilde{S}_{q}^{(2)}(\bc)$.
\begin{lemma}[cf. \cite{Huang} Proposition 2.2]\label{leS2bd}
	We have, uniformly for all $q,N\in\BN$ and $\bc\in\BZ^3$,
	\begin{equation*}
		\widetilde{S}_{q}^{(2)}(\bc)\ll q_2^{\frac{5}{2}}.
	\end{equation*}
\end{lemma}
	We also have the following results on averaging $\widetilde{S}_{q}(\bc)$ over $q$.
	\begin{lemma}[cf. \cite{Huang} Proposition 2.3]\label{le:Kloosterman}
		For any fixed $\kappa\geqslant 0$, uniformly for any $\bc\in\BZ^3$, we have
			\begin{equation*}
			\sum_{\substack{q\leqslant X\\q_{N\Omega}>X^\kappa}}\frac{\left|\widetilde{S}_{q}(\bc)\right|}{q^{2}}\ll_\varepsilon\Upsilon_{\frac{1}{4}}(N)X^{1-\frac{1}{4}\kappa+\varepsilon}.
		\end{equation*}
	\end{lemma}

We define for $\bc\in\BZ^3,l,x\in\BN$ with $\gcd(x,lL)=1$, \begin{equation}\label{eq:CSq2}
	\CS_{l}(x;\bc):=\sum_{\substack{a\bmod l\\(a,l)=1}}\sum_{\substack{\bbeta\bmod lL^2\\F(\bbeta)\equiv m_0N\bmod L^2\\ \bbeta\equiv \blambda_N\bmod L}}\e_{lL^2}\left(a_2 (F(\bbeta)-m_0N)+\overline{x}\bc\cdot\bbeta\right),
\end{equation}  
For each character $\chi\bmod lL^2$, we also define \begin{equation}\label{eq:CA}
	\CA_{l}(\chi;\bc):=\frac{1}{\phi(lL^2)}\sum_{x\bmod lL^2}\chi(x)^c \CS_{l}(x;\bc).
\end{equation}
Then by orthogonality, \begin{equation}\label{eq:CSCA}
	\CS_{l}(x;\bc)=\sum_{\chi\bmod lL^2}\chi(x)\CA_{l}(\chi;\bc).
\end{equation}
The following estimate is better than the one given by Lemma \ref{leS2bd} when $\chi$ is primitive.
\begin{lemma}\label{le:CAbd}
	We have, uniformly for all $\bc,l,N\chi$, 
	$$\CA_{l}(\chi;\bc)\ll_{\varepsilon} \frac{l^{\frac{43}{16}+\varepsilon}}{|\chi_*|^{\frac{1}{4}}},$$ where  $|\chi_*|$ is the modulus of the primitive character $\chi_*$ inducing $\chi$.
\end{lemma}
The relation \eqref{eq:CSCA} and the computation done in \cite[(2.22)]{Huang} show that
\begin{equation}\label{eq:SCS}
	\begin{split}
			\widetilde{S}_{q}^{(2)}(\bc)&=\e_{q_2L^2}\left(-\overline{q_1}\bc\cdot\blambda_N\right) \CS_{q_2}(q_1;\bc)\\ &=\e_{q_2L^2}\left(-\overline{q_1}\bc\cdot\blambda_N\right)\sum_{\chi\bmod q_2L^2}\chi(x)\CA_{q_2}(\chi;\bc).
	\end{split}
\end{equation}  
\subsection{Oscillatory integrals}
The change of variable $\bt:=\frac{L\bs+\blambda_N}{\sqrt{N}}$ yields
	$$\widetilde{I}_{q}(\bc)= \left(\frac{\sqrt{N}}{L}\right)^3 \e_{qL^2}(\bc\cdot\blambda_N)\widetilde{\CI}_{\frac{q}{Q}}\left(w;\frac{\bc}{L}\right),$$ where for any $r>0$ we define
\begin{equation}\label{eq:CI}
\widetilde{\CI}_r(w;\bb):=\int_{\BR^3}w(\bt)h\left(r,F(\bt)-m_0\right)\e_r\left(-\bb\cdot\bt\right)\operatorname{d}\bt.
\end{equation}
The following gathers useful estimates for $\widetilde{\CI}_r(w;\bb)$.
\begin{lemma}\label{le:easyhardest}
	Uniformly for all $\bb$ and $r$, $$\widetilde{\CI}_{r}(w;\bb)\ll 1\quad \text{``Trivial estimate''}.$$
	If moreover $\bb\neq \boldsymbol{0}$, we have
	$$\widetilde{\CI}_{r}(w;\bb) \begin{cases}
		\ll_{A}  r^{-1}|\bb|^{-A} \text{ for all } A>1;  & \quad \text{``Simple estimate''},\\
		\ll_{\varepsilon} \left(\frac{r}{|\bb|}\right)^{\frac{1}{2}-\varepsilon}; & \quad \text{``Harder estimate''}.
	\end{cases}$$
If $\bb=\boldsymbol{0}$ we have 	$$r\frac{\partial \widetilde{\CI}_{r}(w;\boldsymbol{0})}{\partial r}\ll 1.$$
\end{lemma}

	In contrast to the case treated in \cite{Huang} (which is close to being homogeneous), the estimates for the partial derivative $r\frac{\partial \widetilde{\CI}_{r}(w;\bb)}{\partial r}$ when $\bb\neq\boldsymbol{0}$ is not as good as above. On abbreviating $G(\bt):=F(\bt)-m_0$ and on recycling the integral-by-parts manipulation carried out in \cite[p.182]{H-Bdelta}, there exist functions $h_1,h_2$ of class $\CH$ and appropriate weight functions $w_1,w_2: \BR^3\to\BR$ (defined by \cite[(7.1)]{H-Bdelta}) such that
	$$\frac{\partial\widetilde{\CI}_{r}(w;\bb)}{\partial r}=\int_{\BR^3}\left( r^{-2}h_1\left(r,G(\bt)\right)w_1(\bt)+m_0r^{-3}h_2\left(r,G(\bt)\right)w_2(\bt)\right)\e_r\left(-\bb\cdot\bt\right)\operatorname{d}\bt.$$
	The factor $m_0$ being non-zero, we obtain
	\begin{lemma}\label{le:easyhardpartial}
		If  $\bb\neq \boldsymbol{0}$, we have
			$$r\frac{\partial\widetilde{\CI}_{r}(w;\bb)}{\partial r} \begin{cases}
			\ll_{A}  r^{-2}|\bb|^{-A} \text{ for all } A>1;  & \quad \text{``Simple estimate''},\\
			\ll_{\varepsilon} r^{-1}\left(\frac{r}{|\bb|}\right)^{\frac{1}{2}-\varepsilon}; & \quad \text{``Harder estimate''}.
		\end{cases}$$
	\end{lemma}
Comparing with Lemma \ref{le:easyhardest}, we see that an extra $r^{-1}$ appears, which can be problematic when $r$ is close to zero.

\section{Counting $p_0$-integral points}\label{se:counting}
We assume from now on \eqref{eq:Np0}. The main theorem of this section the following sum over non-zero Poisson variables $\bc$.
\begin{theorem}\label{thm:cneq0sum}
	There exists $\CK_h\in\BC$ such that 
	$$\sum_{\substack{\bc\in\BZ^3\setminus\{\boldsymbol{0}\}}}\sum_{q\ll Q}\frac{\widetilde{S}_{q}(\bc)\widetilde{I}_{q}(w;\bc)}{(qL)^3}=N^\frac{3}{2}\CK_h+O\left(\frac{N^\frac{3}{2}}{{(\log N)^{\frac{1}{4}(1-\frac{\sqrt{2}}{2})-A_0}}}\right),$$ for a certain effectively computable numerical constant $0<A_0<\frac{1}{4}(1-\frac{\sqrt{2}}{2})$.
\end{theorem}

For $\bc=\boldsymbol{0}$ we have
\begin{proposition}\label{prop:c=0}
	 The sum $\sum_{q=1}^\infty\frac{\widetilde{S}_{q}(\boldsymbol{0})\widetilde{I}_{q}(w;\boldsymbol{0})}{(qL)^3}$ equals
	$$\begin{cases}
		\left(\frac{\sqrt{N}}{L^2}\right)^3\left(\widetilde{\CI}(w)L^4\widetilde{\mathfrak{S}}_F(p_0;L,\blambda)\log (\sqrt{N})+a_0\right)+O_\varepsilon(N^{\frac{5}{4}+\varepsilon}) &\text{ if } -m\Delta_F=\square;\\
		\left(\frac{\sqrt{N}}{L^2}\right)^3\left(\widetilde{\CI}(w)L^4\BL(1,\psi_0)\widetilde{\mathfrak{S}}_F(p_0;L,\blambda)\right)+O_\varepsilon(N^{\frac{5}{4}+\varepsilon}) &\text{ if } -m\Delta_F\neq\square;
	\end{cases}$$ for a certain $a_0=O(1)$.
\end{proposition}

The majority of remaining part of this section is devoted to the proof of Theorem \ref{thm:cneq0sum}. We prove Proposition \ref{prop:c=0} in \S\ref{se:c=0}.
\subsection{Preliminary manipulation}
Before going into the treatment of different types of $\bc\in\BZ^3\setminus\{\boldsymbol{0}\}$, we take a closer look at the dependency on $N$ for the quantity $\CS_{q_2}(x;\bc)$ \eqref{eq:CSq2}. For each $q_2$, write $q_2^\flat:=q_{2,p_0}$ and $q_2^\natural:=q_2/q_2^\flat$, so that $(q_2^\flat,q_2^\natural L^2)=1$. We also write 
$$a_2=b_1q_2^\natural+b_2q_2^\flat, \quad b_1\bmod q_2^\flat,b_2\bmod q_2^\natural,$$
$$\bbeta=q_2^\natural L^2\bbeta_1+q_2^\flat \bbeta_2,\quad \bbeta_1\bmod q_2^\flat,\bbeta_2\bmod q_2^\natural L^2.$$
Then we now decompose further  \begin{align*}
	\CS_{q_2}(x;\bc)&=\sum_{\substack{b_1\bmod q_2^\flat\\(b_1,q_2^\flat)=1}}\sum_{\substack{\bbeta_1\bmod q_2^\flat}}\e_{q_2^\flat}\left(b_1 (F(q_2^\natural L\bbeta_1)-m_0p_0^{2h})+\overline{x}\bc\cdot\bbeta_1\right)\\ &\times \sum_{\substack{b_2\bmod q_2^\natural\\(b_2,q_2^\natural)=1}}\sum_{\substack{\bbeta_2\bmod q_2^\natural L^2\\F(q_2^\flat\bbeta_2)\equiv m_0N\bmod L^2\\ q_2^\flat\bbeta_2\equiv \blambda_N\bmod L}}\e_{q_2^\natural L^2}\left(b_2 (F(q_2^\flat\bbeta_2)-m_0p_0^{2h})+\overline{x}\bc\cdot\bbeta_2\right).
\end{align*} 
Making the change of variables $$b_1':=(q_2^\natural L)^2 b_1, \quad \bbeta_2':=\overline{p_0}^{h-\operatorname{ord}_{p_0}(q_2)}\bbeta_2,\quad b_2':=p_0^{2h}b_2,$$ we obtain $$\CS_{q_2}(x;\bc)=\CT^{(1)}_{q_2}(x;\bc)\CT^{(2)}_{q_2}(\overline{p_0}^{h-\operatorname{ord}_{p_0}(q_2)}x;\bc),$$
where \begin{equation}\label{eq:CT1}
	\begin{split}
		\CT^{(1)}_{q_2}(x;\bc):=\sum_{\substack{b_1'\bmod q_2^\flat\\(b_1',q_2^\flat)=1}}\sum_{\substack{\bbeta_1\bmod q_2^\flat}}\e_{q_2^\flat}\left(b_1' F(\bbeta_1)+\overline{x}\bc\cdot\bbeta_1\right),  \end{split}
\end{equation}
\begin{equation}\label{eq:CT2}
	\begin{split}
		\CT^{(2)}_{q_2}(x;\bc):=\sum_{\substack{b_2'\bmod q_2^\natural\\(b_2',q_2^\natural)=1}}\sum_{\substack{\bbeta_2'\bmod q_2^\natural L^2\\F(\bbeta_2')\equiv m_0\bmod L^2\\ \bbeta_2'\equiv \blambda\bmod L}}\e_{q_2^\natural L^2}\left(b_2' (F(\bbeta_2')-m_0)+\overline{x}\bc\cdot\bbeta_2'\right),
	\end{split}
\end{equation}
Note that we always have $\operatorname{ord}_{p_0}(q_2)\leqslant h$.  
For $\chi\bmod q_2L^2$, we write $\chi_1\bmod q_2^\flat,\chi_2\bmod q_2^\natural L^2$ for its decomposition. Recalling \eqref{eq:CA}, we then have
\begin{align}
	\CA_{q_2}(\chi;\bc)&=\frac{1}{\phi(q_2L^2)}\sum_{\substack{x_1\bmod q_2^\flat\\x_2\bmod q_2^\natural L^2}}\chi_1(x_1)^c\CT^{(1)}_{q_2}(x_1;\bc)\chi_2(x_2)^c \CT^{(2)}_{q_2}(\overline{p_0}^{h-\operatorname{ord}_{p_0}(q_2)}x_2;\bc)\nonumber\\
	&=\chi_2(p_0^h)\widetilde{\CA}_{q_2}(\chi;\bc).\label{eq:CAtildeCA}
\end{align}
where we introduce \begin{equation}\label{eq:CAtilde}
	\widetilde{\CA}_{q_2}(\chi;\bc):=\frac{\chi_2(q_2^\flat)^c}{\phi(q_2L^2)}\sum_{\substack{x_1\bmod q_2^\flat\\x_2\bmod q_2^\natural L^2}}\chi(x_1x_2)^c \CT^{(1)}_{q_2}(x_1;\bc)\CT^{(2)}_{q_2}(x_2;\bc).
\end{equation} Note that $\widetilde{\CA}_{q_2}(\chi;\bc)$ depends purely on the quadratic form $F$ and the local conditions (in addition to $\bc,q_2,\chi$) but is independent of $h$.
By Lemma \ref{le:CAbd}, we have, uniformly for all $\bc,q_2,\chi$,
\begin{equation}\label{eq:CAbd}
	\widetilde{\CA}_{q_2}(\chi;\bc)\ll_{\varepsilon} \frac{q_2^{\frac{43}{16}+\varepsilon}}{|\chi_*|^{\frac{1}{4}}}.
\end{equation}

As in \cite[\S5]{Huang}, we introduce
\begin{equation}\label{eq:CGqc}
	\widetilde{\CG}_q(\bc):=\frac{\e_{qL^2}(\bc\cdot\blambda_N)\widetilde{S}_{q}(\bc)\widetilde{\CI}_{\frac{q}{Q}}(w;\frac{\bc}{L})}{q^3}.
\end{equation}
so that the sum in Theorem \ref{thm:cneq0sum} equals \begin{equation}\label{eq:execstep}
	\frac{N^{\frac{3}{2}}}{L^6}\sum_{\substack{\bc\in\BZ^3\setminus\{\boldsymbol{0}\}}}\sum_{q\ll Q}\widetilde{\CG}_q(\bc).
\end{equation}
Then by Lemma \ref{le:S1} and the equations \eqref{eq:SCS} \eqref{eq:CAtildeCA} (cf. the manipulation in \cite[\S4.1]{Huang}), we have
\begin{equation}\label{eq:CGq}
	\widetilde{\CG}_q(\bc)=\left(\sum_{\chi\bmod q_2L^2}\frac{\psi_0(q_1)\chi(q_1)\chi_2(p_0^h)\widetilde{\CA}_{q_2}(\chi;\bc)}{q_1q_2^3}\sum_{\substack{u\bmod q_1\\ G_{\bc,q_2}(u)\equiv 0\bmod q_1}}\e_{q_1}(u)\right)\widetilde{\CI}_{\frac{q}{Q}}\left(w,\frac{\bc}{L}\right).
\end{equation}

\subsection{Contribution from the exceptional $\bc$}\label{se:exec}
We shall call a Poisson variable $\bc\in\BZ^3\setminus\{\boldsymbol{0}\}$ \emph{exceptional} if the corresponding polynomial  $G_{\bc,l}(T)$ \eqref{eq:polyG} is reducible, or equivalently, \begin{equation}\label{eq:square}
	m_0 \Delta_F F^*(\bc)=\square.
\end{equation}
We shall also call an exceptional $\bc$ of \emph{Type I} (resp. \emph{Type II}) if moreover $F^*(\bc)\neq 0$ (resp. if $F^*(\bc)=0$).
\begin{theorem}\label{thm:exec}
	For all exceptional $\bc\in\BZ^3\setminus\{\boldsymbol{0}\}$, we have
	$$\sum_{q\ll Q} \widetilde{\CG}_q(\bc)=\widetilde{\eta}_h(\bc)+O_\varepsilon\left(|\bc|^{\frac{1}{2}+\varepsilon}Q^{-2\lambda_0+\varepsilon}\right),$$ where $\lambda_0>0$ is a certain effectively computable numerical constant, and $\eta_h(\bc)$ is defined respectively (for Type I and II) by \eqref{eq:etac} and \eqref{eq:etac2}, and satisfies that for all $M>0$, \begin{equation}\label{eq:etacbd}
		\widetilde{\eta}_h(\bc)\ll_M |\bc|^{-\frac{1}{3}M}.
	\end{equation}
\end{theorem}

\begin{proof}[Proof of Theorem \ref{thm:exec}]
Let $0<\delta,\kappa<1$ be fixed.
By Lemmas \ref{le:Kloosterman} and the ``harder estimate'' in \ref{le:easyhardest} (cf. the proof of \cite[Lemma 5.3]{Huang}) we may restrict to the range $Q^{1-\delta}<q\ll Q$, $\delta>0$, upon adding up an error term $\ll_{\varepsilon} |\bc|^{-\frac{1}{2}+\varepsilon} Q^{-\frac{1}{2}\delta+\varepsilon}$.
On the other hand, by Lemma \ref{le:Kloosterman} and the harder estimate as in \cite[(5.5)]{Huang}, the contribution from $q_{m_0p_0\Omega}>X^\kappa$ is $\ll_{\varepsilon}|\bc|^{\frac{1}{2}+\varepsilon} Q^{-\frac{1}{2}\kappa+\varepsilon}$. 
We may therefore proceed under the assumption that $Q^{1-\delta}<q\ll Q$ and $q_{m_0p_0\Omega}\leqslant X^\kappa$, where we shall determine the values of $\delta,\kappa>0$ later. 

We first assume that $\bc$ is exceptional of Type I.
By  \eqref{eq:CGq}, we can write
\begin{equation}\label{eq:sumCG}
\sum_{\substack{Q^{1-\delta}<q\ll Q\\q_{m_0p_0\Omega}\leqslant X^\kappa}} \widetilde{\CG}_q(\bc)=\sum_{\substack{q_2\leqslant Q^\kappa \\q_2\mid(m_0p_0\Omega)^\infty}}\frac{1}{q_2^3}\sum_{\chi\bmod q_2L^2}\chi_2(p_0^h)\widetilde{\CA}_{q_2}(\chi;\bc)
 \CV_{q_2,\bc}(\chi;Q), 
\end{equation}
where for $T\geqslant Q^{1-\delta}$, we introduce
\begin{equation}\label{eq:CU}
	\CV_{q_2,\bc}(\chi;T):=\sum_{\substack{\frac{Q^{1-\delta}}{q_2}<q_1\ll \frac{T}{q_2}\\ \gcd(q_1,m_0p_0\Omega)=1}}\psi_0(q_1)\chi(q_1)\sum_{\substack{u\bmod q_1\\ G_{\bc,q_2}(u)\equiv 0\bmod q_1}}\e_{q_1}(u)\frac{\widetilde{\CI}_{q_1q_2/Q}(w,\frac{\bc}{L})}{q_1}
\end{equation}

Let $$\CN(\bc):=	\sqrt{m_0 \Delta_F F^*(\bc)},$$
	\begin{equation}\label{eq:a0b0}
	a_0:=\frac{\Delta_F}{\gcd(\Delta_F,p_0^hq_2L^2\CN(\bc))},\quad b_0:=\frac{p_0^h q_2L^2\CN(\bc)}{\gcd(\Delta_F,p_0^hq_2L^2\CN(\bc))}.
\end{equation}
The argument in \cite[\S3.2]{Huang} 
shows that the $q_1$-sum  $\CV_{q_2,\bc}(\chi;T)$ can be dissected as
\begin{multline}\label{eq:sumgjk}
	\CV_{q_2,\bc}(\chi;T)=2\sum_{\substack{g\mid 2a_0b_0\\ (g,a_0)=1}}\mu^2(g)\sum_{\substack{k\leqslant\sqrt{\frac{T}{q_2g}}\\ (k,a_0b_0 m_0p_0\Omega)=1}}\psi_0(k)\chi(k)\\ \times\sum_{\substack{\frac{Q^{1-\delta}}{q_2gk}<j\ll\frac{T}{q_2gk}\\ (j,ka_0b_0 m_0p_0\Omega)=1}}\psi_0(j)\chi(j)\e_{a_0k}\left(\widehat{t}\cdot\overline{j}\right)\e_{a_0gjk}(b_0)\frac{\widetilde{\CI}_{gjkq_2/Q}(w,\frac{\bc}{L})}{gjk},
\end{multline}
where  $\widehat{t}\bmod a_0k$ depends on $g,k,a_0,b_0,\bc$ but does not depend on $j$ and satisfies  $\gcd(\widehat{t},a_0k)=1$.

On fixing $k$ and $\chi$, let $$\CB_{\chi}(k;Y):=\sum_{\substack{j\leqslant Y\\ (j,ka_0b_0 m_0p_0\Omega)=1}}\psi_0(j)\chi(j)\e_{a_0k}\left(\widehat{t}\cdot\overline{j}\right).$$
Then \cite[Lemma 3.2]{Huang} shows the existence of $\Theta_{\chi}(k)\in\BC$ such that
\begin{equation}\label{eq:CBk}
	\CB_{\chi}(k;Y)= \Theta_{\chi}(k)Y+O_\varepsilon\left((a_0b_0)^\varepsilon|\chi|^{\frac{11}{8}+\varepsilon}(a_0k)^{\frac{7}{8}+\varepsilon}\right).
\end{equation}
By partial summation, we have (the gcd condition can be detected by a certain principal character) that the inner $j$-sum in \eqref{eq:sumgjk} equals
\begin{multline*}
	-\int_{Q^{1-\delta}/q_2gk}^{cQ/q_2gk}\CB_{\chi}(k;t)\frac{\partial}{\partial t}\left(\e_{a_0gkt}(b_0)\frac{\widetilde{\CI}_{gktq_2/Q}(w,\frac{\bc}{L})}{gkt}\right)\operatorname{d}t\\+\left[\CB_{\chi}(k;t)\e_{a_0gkt}(b_0)\frac{\widetilde{\CI}_{gktq_2/Q}(w,\frac{\bc}{L})}{gkt}\right]_{Q^{1-\delta}/q_2gk}^{cQ/q_2gk}.
\end{multline*}
Here we fix $c>0$ such that $\widetilde{\CI}_{r}(w,\bb)$ vanishes for $r\geqslant c$. By Lemma \ref{le:easyhardest}, the variation term, evaluated to zero at $t=cQ/q_2gk$, is thus
$$\ll_\varepsilon Q^{-\frac{1}{2}\delta+\varepsilon}\frac{1}{|\bc|^{\frac{1}{2}-\varepsilon}gk}.$$
On employing \eqref{eq:CBk} and on setting $t\mapsto s:=tq_2gk$, the integral above equals
\begin{multline*}
	-\frac{\Theta_{\chi}(k)}{gk}\int_{Q^{1-\delta}}^{cQ}s\frac{\partial}{\partial s}\left(\e_{a_0s}(b_0q_2)\frac{\widetilde{\CI}_{s/Q}(w,\frac{\bc}{L})}{s}\right)\operatorname{d}s\\ +O_\varepsilon \left(Q^\varepsilon q_2^{\frac{19}{8}+\varepsilon}k^{\frac{7}{8}+\varepsilon}\left|\int_{Q^{1-\delta}}^{cQ}\frac{\partial}{\partial s}\left(\e_{a_0s}(b_0q_2)\frac{\widetilde{\CI}_{s/Q}(w,\frac{\bc}{L})}{s}\right)\right|\operatorname{d}s\right).\end{multline*}
Recall \eqref{eq:a0b0}. Executing the change of variable $s\mapsto r:=\frac{s}{Q}$ by using the ``harder estimate'' in Lemma \ref{le:easyhardest} several times, the integral in the main term can be computed as follows, via integration-by-parts.
\begin{align*}
	&-\int_{Q^{-\delta}}^{c}r\frac{\partial}{\partial r}\left(\e_{\Delta_F r}(q_2^2L^3\CN(\bc))\frac{\widetilde{\CI}_{r}(w,\frac{\bc}{L})}{r}\right)\operatorname{d}r\\ =&\int_{Q^{-\delta}}^{c} \e_{\Delta_F r}(q_2^2L^3\CN(\bc))\frac{\widetilde{\CI}_{r}(w,\frac{\bc}{L})}{r}\operatorname{d}r+O_\varepsilon (Q^{-\frac{1}{2}\delta+\varepsilon}|\bc|^{-\frac{1}{2}+\varepsilon})\\ =&\widetilde{\CJ}_{q_2}(\bc)+O_\varepsilon (Q^{-\frac{1}{2}\delta+\varepsilon}|\bc|^{-\frac{1}{2}+\varepsilon}),
\end{align*} where we introduce \begin{equation}\label{eq:Lambda}
\widetilde{\CJ}_{u}(\bc):=\int_{0}^{c} \e_{\Delta_F r}(u^2L^3\CN(\bc))\frac{\widetilde{\CI}_{r}(w,\frac{\bc}{L})}{r}\operatorname{d}r.
\end{equation} 
Note that, using the ``simple estimate'' and ``harder estimate'' in Lemma \ref{le:easyhardest} we can show, as in the proof of \cite[Theorem 5.2]{Huang}, uniformly for all $u$,
\begin{equation}\label{eq:CJubd}
	\widetilde{\CJ}_{u}(\bc)\ll_M |\bc|^{-\frac{1}{3}M},
\end{equation} for any $M>0$.

 The integral in the error term can be bounded as follows, using the ``harder estimate'' in Lemmas \ref{le:easyhardest} \ref{le:easyhardpartial},
\begin{align*}
	\ll &Q^{-1}\int_{Q^{-\delta}}^{c}\left|\frac{\partial}{\partial r}\left(\e_{\Delta_F r}(q_2^2L^3\CN(\bc))\frac{\widetilde{\CI}_{r}(w,\frac{\bc}{L})}{r}\right)\right|\operatorname{d}r\\ \ll &Q^{-1} q_2^2 \CN(\bc)\int_{Q^{-\delta}}^{c}\left(\left|\frac{\widetilde{\CI}_{r}(w,\frac{\bc}{L})}{r^3}\right|+\left|\frac{\frac{\partial}{\partial r}\widetilde{\CI}_{r}(w,\frac{\bc}{L})}{r}\right|\right)\operatorname{d}r\ll_{\varepsilon} q_2^2|\bc|^{\frac{1}{2}+\varepsilon}Q^{-1+\frac{3}{2}\delta+\varepsilon}.
\end{align*}

We come back to the remaining $g,k$-sums in $\CV_{q_2,\bc}(\chi;T)$. The error term arising from the previous computations is bounded by
$$\ll_{\varepsilon} Q^\varepsilon |\bc|^{\frac{1}{2}+\varepsilon} \left(q_2^{\frac{55}{16}}Q^{-\frac{1}{16}+\frac{3}{2}\delta}+Q^{-\frac{1}{2}\delta}\right).$$
 We introduce  (cf. \cite[(3.16)]{Huang})
\begin{equation}\label{eq:Gammachi}
	\Gamma_{\chi}(q_2;\bc):=2\sum_{\substack{g\mid 2a_0b_0\\ (g,a_0)=1}}\frac{\mu^2(g)}{g}\sum_{\substack{k\in\BN\\(k,a_0b_0 m_0p_0\Omega)=1}}\frac{\psi_0(k)\chi(k)\Theta_{\chi}(k)}{k},
\end{equation}  which satisfies \begin{equation}\label{eq:Gammachibd}
\Gamma_{\chi}(q_2;\bc)\ll_{\varepsilon}\frac{(q_2|\bc|)^\varepsilon}{|\chi_*|^{\frac{15}{32}}}.
\end{equation} Mimicking the proof of \cite[Theorem 3.1 (2)]{Huang}, extending the $k$-sum to infinity produces an error term of order $\ll_{\varepsilon} q_2^{\frac{1}{16}+\varepsilon} Q^{-\frac{1}{16}+\varepsilon}$. We conclude that 
$$\CV_{q_2,\bc}(\chi;T)=\Gamma_{\chi}(q_2;\bc)\widetilde{\CJ}_{q_2}(\bc)+O_\varepsilon\left(Q^\varepsilon |\bc|^{\frac{1}{2}+\varepsilon} \left(q_2^{\frac{55}{16}}Q^{-\frac{1}{16}+\frac{3}{2}\delta}+Q^{-\frac{1}{2}\delta}\right)\right).$$

Going back to the $q$-sum \eqref{eq:sumCG} and applying the argument of \cite[(4.5)]{Huang} based on the bounds \eqref{le:CAbd} \eqref{eq:Gammachibd}, the main term takes the form $$\sum_{\substack{u\leqslant Q^\kappa\\u\mid (m_0p_0\Omega)^\infty}}\sum_{\chi\bmod uL^2}\frac{\chi_2(p_0^h)\widetilde{\CA}_{u}(\chi;\bc)\Gamma_{\chi}(u;\bc)\widetilde{\CJ}_{u}(\bc)}{u^3}=\widetilde{\eta}_h(\bc)+O_\varepsilon(|\bc|^\varepsilon Q^{-\frac{1}{32}\kappa+\varepsilon}),$$ 
where we let
\begin{equation}\label{eq:etac}
	\widetilde{\eta}_h(\bc):=\sum_{u\mid (m_0p_0\Omega)^\infty}\sum_{\chi\bmod uL^2}\frac{\chi_2(p_0^h)\widetilde{\CA}_{u}(\chi;\bc)\Gamma_{\chi}(u;\bc)\widetilde{\CJ}_{u}(\bc)}{u^3}.
\end{equation}
The error term contributes all together $$\ll_{\varepsilon} Q^\varepsilon |\bc|^{\frac{1}{2}+\varepsilon} \left(Q^{-\frac{1}{16}+\frac{55}{16}\kappa+\frac{3}{2}\delta}+Q^{-\frac{1}{2}\delta}\right).$$

We next assume $\bc$ to be exceptional of Type II.
Then by Lemma \ref{le:S1} (cf. \cite[(4.7)]{Huang}), $$\sum_{\substack{u\bmod q_1\\ G_{\bc,q_2}(u)\equiv 0\bmod q_1}}\e_{q_1}(u)=\sum_{\substack{u^2\equiv 0\bmod q_1}} \e_{q_1}(u)=\mu(q_1)^2,$$ and hence
$$\CV_{q_2,\bc}(\chi;T)=\sum_{\substack{\frac{Q^{1-\delta}}{q_2}<q_1\ll \frac{T}{q_2}\\ \gcd(q_1,m_0p_0\Omega)=1}}\psi_0(q_1)\chi(q_1)\mu^2(q_1)\frac{\widetilde{\CI}_{q_1q_2/Q}(w,\frac{\bc}{L})}{q_1}.$$
For fixed $\chi$ let us define $$B_{\chi}(t):=\sum_{\substack{q\leqslant t\\\gcd(q,m_0p_0\Omega)=1}}\psi_0(q)\chi(q)\mu^2(q)$$
If $\chi\psi_0$ is non-principal, by the Burgess bound as in \cite[\S4.1.3]{Huang}, we have $$B_{\chi}(t)\ll_{\varepsilon} Q^\varepsilon t^{\frac{1}{2}+\varepsilon}q_2^{\frac{3}{16}+\varepsilon}.$$ Coupled with partial summation  we obtain that for such a $\chi$,
\begin{align*}
	\CV_{q_2,\bc}(\chi;T)&=-\int_{Q^{1-\delta}/q_2}^{cQ/q_2}B_{\chi}(t)\frac{\partial}{\partial t}\left(\frac{\widetilde{\CI}_{tq_2/Q}(w,\frac{\bc}{L})}{t}\right)\operatorname{d}t+\left[B_{\chi}(t)\frac{\widetilde{\CI}_{tq_2/Q}(w,\frac{\bc}{L})}{t}\right]_{Q^{1-\delta}/q_2}^{cQ/q_2}\\ &\ll_{\varepsilon} Q^\varepsilon q_2^{\frac{3}{16}}\left(\left(\frac{Q}{q_2}\right)^{-\frac{1}{2}+\varepsilon}\int_{Q^{-\delta}}^{c}r^{\frac{1}{2}+\varepsilon}\left|\frac{\partial}{\partial r}\left(\frac{\widetilde{\CI}_{r}(w,\frac{\bc}{L})}{r}\right)\right|\operatorname{d}r+\left(\frac{Q^{1-\delta}}{q_2}\right)^{-\frac{1}{2}+\varepsilon}\right)\\ &\ll_{\varepsilon} Q^{-\frac{1}{2}+2\delta+\varepsilon}q_2^{\frac{11}{16}}|\bc|^{-\frac{1}{2}+\varepsilon},
\end{align*} by Lemma \ref{le:easyhardpartial}.

For $\chi=\psi_0$, we have
$$B_{\psi_0}(t)=\frac{6}{\pi^2}\prod_{p\mid p_0m_0\Omega}\left(1-\frac{1}{p}\right)t+O(t^\frac{1}{2}),$$ whence again by partial summation, $$\CV_{q_2,\bc}(\psi_0;T)=\CJ(\bc)\frac{6}{\pi^2}\prod_{p\mid p_0m_0\Omega}\left(1-\frac{1}{p}\right)+O_\varepsilon\left(Q^{-\frac{1}{2}\delta+\varepsilon}|\bc|^{-\frac{1}{2}+\varepsilon}\right),$$ where we define (as in \cite[(5.1)]{Huang})\begin{equation}\label{eq:CJbc}
	\CJ(\bc):=\int_{0}^{c}\frac{\widetilde{\CI}_{r}(w,\frac{\bc}{L})}{r}\operatorname{d}r.
\end{equation} It satisfies the same kind of bound as \eqref{eq:CJubd}.
We define for such an exceptional  $\bc$ of type II,
\begin{equation}\label{eq:etac2}
	\widetilde{\eta}_h(\bc):=\CJ(\bc)\sum_{\substack{u\mid (m_0p_0\Omega)^\infty\\\exists\chi\bmod uL^2,\chi\psi_0 \text{ principal}}}\frac{\psi_{0,2}(p_0^h)\widetilde{\CA}_{u}(\psi_0;\bc)}{u^3}\frac{6}{\pi^2}\prod_{p\mid p_0m_0\Omega}\left(1-\frac{1}{p}\right).
\end{equation} Similarly the $q_2$-sum converges to $\widetilde{\eta}_h(\bc)$ up to $O_\varepsilon(|\bc|^\varepsilon Q^{-\frac{1}{32}\kappa+\varepsilon})$.

We therefore finish the proof by fixing $\delta,\kappa$ small enough. The rapid decay for $\widetilde{\eta}_h(\bc)$ \eqref{eq:etacbd} results from that for $\CJ_{u}(\bc)$ and $\CJ(\bc)$. \end{proof}

\begin{remark}\label{rmk:exceptionalc}
	Let us now explain the main difference between \cite[Proof of Theorem 5.2]{Huang} and here.  Recall the inner $j$-sum in \eqref{eq:sumgjk}. Since the term $b_0$ can be as large as $Q$, we see that the factor $\e_{a_0gjk}(b_0)$ can be highly oscillatory. This is the reason why we need to bring the oscillatory integral in and couple it with $\e_{a_0gjk}(b_0)$ when doing partial summation as early as dissecting the sum $\CV_{q_2,\bc}(\chi;T)$.
\end{remark}

		\subsection{Estimation of sums attached to ordinary $\bc$}\label{se:ordc}
		We shall call a $\bc\in\BZ^3\setminus\{\boldsymbol{0}\}$ \emph{ordinary} if \begin{equation}\label{eq:cord}
			m_0 \Delta_F F^*(\bc)\neq\square.
		\end{equation} 
		
		We recall \eqref{eq:CGqc}.
		\begin{theorem}\label{thm:cord}
		For any $M>0$, we have, uniformly for any ordinary $\bc\in\BZ^3\setminus\{\boldsymbol{0}\}$, 
			$$\sum_{q\ll Q}\widetilde{\CG}_q(\bc)\ll_M |\bc|^{-M}(\log N)^{-\frac{1}{4}(1-\frac{\sqrt{2}}{2})+A_0},$$ where $0<A_0<\frac{1}{4}(1-\frac{\sqrt{2}}{2})$ is a  certain effectively computable numerical constant.
		\end{theorem}

Before going into the proof of Theorem \ref{thm:cord}, we collect various ingredients on sums of Salié sums established in \cite{Huang}.
			For any $\bt\in\BR^3, \bc\in\BZ^3$, let \begin{equation}\label{eq:CFtc}
				\CF_{\bt,\bc}(X):=\sum_{\substack{q\leqslant X}}\frac{\e_{qL^2}(\bc\cdot\blambda_N)\widetilde{S}_{q}(\bc)\e_{\frac{q}{Q}}\left(-\frac{\bc}{L}\cdot\bt\right)}{q^2}.
			\end{equation} 			The following estimate is analogous to \cite[Theorem 4.1]{Huang}.
			\begin{proposition}\label{prop:sumSalie}
			Uniformly for any $\bt\in\BR^3$ and for any ordinary $\bc\in\BZ^3\setminus\{\boldsymbol{0}\}$, we have 
				$$\CF_{\bt,\bc}(X)\ll_{\varepsilon} |\bc|^{2+\varepsilon}\frac{X}{(\log X)^{\frac{1}{4}(1-\frac{\sqrt{2}}{2})-\varepsilon}}.$$ 
			\end{proposition}
		\begin{remark}
			For comparison, a stronger Linnik-type conjecture concerning sums of ``twisted Kloosterman sums'' is proposed in \cite[Conjecture 1.1]{B-K-S}. Proposition \ref{prop:sumSalie}, sufficient for the our purpose, may be viewed as an unconditional non-trivial bound for sums of ``twisted Salié sums''.
		\end{remark}
			The proof of Proposition \ref{prop:sumSalie} relies on the following variation of \cite[Proposition 2.3 \& Theorem 3.1 (1)]{Huang}.
			\begin{lemma}\label{le:Kloosterman2}
			Let $\kappa\geqslant 0$.	Uniformly for all ordinary $\bc$, we have
				$$\sum_{\substack{q\leqslant X\\ q_{m_0p_0\Omega}>(\log X)^\kappa}}\frac{\left|\widetilde{S}_{q}(\bc)\right|}{q^{2}}\ll_{\varepsilon}  |\bc|^{2+\varepsilon} X(\log X)^{-\frac{1}{4}\kappa}.$$
			\end{lemma}
			\begin{lemma}\label{le:sumSalie2}
				Let $\hbar:\BZ\to\BC$ be a function such that $|\hbar|\leqslant 1$. Consider
				\begin{equation}\label{eq:U}
					\CU_{\bc,l}(\hbar;X):=\sum_{\substack{n\leqslant X\\ \gcd(n,m_0p_0\Omega)=1}}\hbar(n)\sum_{\substack{v\bmod n\\ G_{\bc,l}(v)\equiv 0\bmod n}}\e_n(v),
				\end{equation} where we recall the polynomial $G_{\bc,l}(T)$ \eqref{eq:polyG}. Then uniformly for such $\hbar,l$ and ordinary $\bc$, we have
			$$\CU_{\bc,l}(\hbar;X)\ll_{\varepsilon} |\bc|^{2+\varepsilon} \frac{X}{(\log X)^{\frac{1}{4}(1-\frac{\sqrt{2}}{2})-\varepsilon}}.$$
			\end{lemma}
		
	The proofs of \cite[Proposition 2.3 \& Theorem 3.1 (1)]{Huang} can readily be modified to be adapted to Lemmas \ref{le:Kloosterman2} and \ref{le:sumSalie2} upon considering  $$\varrho_{\bc}(n):=\#\{v\bmod n:v^2\equiv m_0\Delta_FF^*(\bc)\bmod n\},$$ and $$S_{\bc}(u,n):=\sum_{\substack{v\bmod n\\ v^2\equiv m_0\Delta_FF^*(\bc)\bmod n}}\e_{n}(uv),$$ 
	so that $$\CU_{\bc,l}(\hbar;X)=\sum_{\substack{n\leqslant X\\ \gcd(n,m_0p_0\Omega)=1}}\hbar(n)S_{\bc}(\overline{\Delta_F}lL^2p_0^h).$$
	Note that $\Upsilon_1(m_0N)=O(1)$. We omit the details.
			\begin{proof}[Sketch of proof of Proposition \ref{prop:sumSalie}]
			The proof is analogous to that of \cite[Theorem 4.1 (1)]{Huang}. Let $\CF_{\bt,\bc}^{(1)}(\beta;X)$ denote the $q$-sum in \eqref{eq:CFtc} with the extra condition $q_{m_0p_0\Omega}\leqslant (\log X)^\beta$. 
				Lemma \ref{le:Kloosterman2} implies that
				$$\CF_{\bt,\bc}(X)-\CF_{\bt,\bc}^{(1)}(\beta;X)\ll_\varepsilon |\bc|^{2+\varepsilon} \frac{X}{(\log X)^{\frac{\beta}{4}}}.$$
				
				For $u\in\BN$ and for any Dirichlet character $\chi\bmod uL^2$, we define $\hbar:\BN\to\BC$ by $$\hbar_u(n):=\psi_0(n)\chi(n)\e_{un/Q}\left(-\frac{\bc}{L}\cdot\bt\right).$$ Now we can decompose in similar way to \eqref{eq:CGq} (cf. also \cite[\S4.1]{Huang})
				$$\CF_{\bt,\bc}^{(1)}(\beta;X)=\sum_{\substack{q_2\leqslant (\log X)^\beta\\ q_2\mid(m_0p_0\Omega)^\infty}}\frac{1}{q_2^2}\sum_{\chi\bmod q_2L^2}\chi_2(p_0^h)\widetilde{\CA}_{q_2}(\chi;\bc)\CU_{\bc,q_2}\left(\hbar_{q_2},\frac{X}{q_2}\right).$$ Using directly Lemma \ref{leS2bd}, we have $$\widetilde{\CA}_{q_2}(\chi;\bc)\ll q_2^\frac{5}{2}.$$ Then by Lemma \ref{le:sumSalie2}, we have
				\begin{align*}
					\CF_{\bt,\bc}^{(1)}(\beta;X)\ll_\varepsilon|\bc|^{2+\varepsilon} \frac{X}{(\log X)^{1-\frac{\sqrt{2}}{2}-\varepsilon}}\sum_{\substack{q_2\leqslant (\log X)^\beta\\ q_2\mid(m_0p_0\Omega)^\infty}}q_2^{\frac{1}{2}}\ll  |\bc|^{2+\varepsilon} \frac{X}{(\log X)^{1-\frac{\sqrt{2}}{2}-\frac{3}{4}\beta-\varepsilon}}.
				\end{align*} 
				It remains to choose $\beta=1-\frac{\sqrt{2}}{2}$ to conclude.
			\end{proof}
	
\begin{proof}[Proof of Theorem \ref{thm:cord}]

We need to distinguish several ranges determined by the relative sizes of $q$ and $\bc$. 	 Let us fix $M_0\in\BN$ a reasonably large integer and $\varepsilon_0>0$ sufficiently small, whose values are to be specified at the end of the proof. 
			
			First consider the range $$q\in I_0:=\left[Q\min\left(|\bc|^{-M_0},(\log Q)^{-\varepsilon_0}\right),cQ\right),$$ where we recall that $c>0$ is fixed such that the function $r\mapsto \widetilde{\CI}_{r}\left(w;\frac{\bc}{L}\right)$ is supported in $(0,c)$. 
			We then have
\begin{equation}\label{eq:qsumI}
				\sum_{q\in I_0} \widetilde{\CG}_q(\bc)	=-\int_{I_0}\CF_{\boldsymbol{0},\bc}(t)\frac{\partial}{\partial t}\left(\frac{\widetilde{\CI}_{\frac{t}{Q}}\left(w;\frac{\bc}{L}\right)}{t}\right)\operatorname{d}t+\left[\CF_{\boldsymbol{0},\bc}(t)\frac{\widetilde{\CI}_{\frac{t}{Q}}\left(w;\frac{\bc}{L}\right)}{t}\right]_{\frac{Q}{\max\left(|\bc|^{M_0},(\log Q)^{\varepsilon_0}\right)}}^{cQ},
\end{equation} where we recall \eqref{eq:CFtc}.
  The argument of \cite[Proof of Proposition 5.5]{Huang} may be recycled, using only the ``simple estimate'' for the partial derivative of $\widetilde{\CI}_r$ in Lemma \ref{le:easyhardpartial} and also Lemma \ref{le:Kloosterman2} (with no condition on $q_{m_0p_0\Omega}$), the integral in \eqref{eq:qsumI} can be estimated as follows: 
			\begin{align*}
				\ll_\varepsilon &\frac{|\bc|^{2+\varepsilon}}{(\log Q)^{\frac{1}{4}(1-\frac{\sqrt{2}}{2})-\varepsilon}}\int_{I_0}t\left|\frac{\partial}{\partial t}\left(\frac{\widetilde{\CI}_{\frac{t}{Q}}\left(w;\frac{\bc}{L}\right)}{t}\right)\right|\operatorname{d}t\\ \ll_M &\frac{|\bc|^{2-M+\varepsilon}}{(\log Q)^{\frac{1}{4}(1-\frac{\sqrt{2}}{2})-\varepsilon}}\int_{I_0/Q}\frac{\operatorname{d}r}{r^2}\\ \ll & \frac{|\bc|^{2-M+\varepsilon}}{(\log Q)^{\frac{1}{4}(1-\frac{\sqrt{2}}{2})-\varepsilon}}\max\left(|\bc|^{M_0},(\log Q)^{\varepsilon_0}\right)\ll \frac{|\bc|^{2+M_0-M+\varepsilon}}{(\log Q)^{\frac{1}{4}(1-\frac{\sqrt{2}}{2})-\varepsilon_0-\varepsilon}}.
			\end{align*}
			Again using the simple estimate for $\widetilde{\CI}_r$ in Lemma \ref{le:easyhardest}, the variation term is easily seen to have at most the same order of magnitude as above. 
It remains to adjust the choice of $M$ in terms of $M_0$.
			
			We now focus on the range 
			\begin{equation}\label{eq:condkey}
			q\in I_1:=\left[1, Q\min\left(|\bc|^{-M_0},(\log Q)^{-\varepsilon_0}\right)\right),
			\end{equation} which necessitates a more demanding analysis, and it is from now on our argument diverges from \cite[Proof of Proposition 5.5]{Huang}. Our approach is a refinement of Heath-Brown's stationary phase analysis \cite[\S8]{H-Bdelta}. 
		
Fix $r>0$. On defining \begin{equation}\label{eq:Prt}
	P_{r}(t):=\int_{\BR}w_0(v)h(r,v)\e(-tv)\operatorname{d}v,
\end{equation} where $w_0$ is an appropriate weight function depending on $w$, by Fourier inversion, we can express the oscillatory integral as
$$\widetilde{\CI}_{r}\left(w;\bb\right)=\int_{\BR}P_{r}(t)\operatorname{d}t\int_{\BR^3}w_1(\bx)\e\left(tG(\bx)-\frac{\bb}{r}\cdot\bx\right)\operatorname{d}\bx,$$ where we write $$G(\bx):=F(\bx)-m_0$$ and $w_1(\bx)$ is another weight function (cf. \cite[p. 184]{H-Bdelta}).
We gather a few helpful estimates for \eqref{eq:Prt}.
\begin{lemma}\label{le:Prt}
	Uniformly for all $t\in\BR$, we have $$\max\left(P_r(t), r\frac{\partial P_r(t)}{\partial r}\right)\ll 1.$$ Moreover,
	$$\max\left(r\int_{\BR}|P_r(t)|\operatorname{d}t,r^2\int_{\BR}\left|\frac{\partial}{\partial r}P_r(t)\right|\operatorname{d}t\right)\ll 1.$$
\end{lemma}
\begin{proof}[Sketch of proof of Lemma \ref{le:Prt}]
	Indeed, as shown by \cite[Lemma 5]{H-Bdelta} that for $k=0,1$, the function $v\mapsto r^{k+1}\frac{\partial^k}{\partial r^k}h(r,v)$ is of class $\CH$. Also by \cite[Lemma 17 (7.3)]{H-Bdelta}, we have
	 $$\int_{\BR} \left|P_{r}(t)\right|\operatorname{d}t \ll \int_{0}^{\frac{1}{r}}\operatorname{d}t+\frac{1}{r^2}\int_{\frac{1}{r}}^{\infty} \frac{1}{t^2}\operatorname{d}t\ll \frac{1}{r}.$$ Similarly we can deduce the corresponding estimates for the partial derivative of $P_r(t)$.
\end{proof}

In what follows for any fixed $r=\frac{q}{Q}>0$ and ordinary $\bc$, we write \begin{equation}\label{eq:u}
	\bu_r=\bu_r(\bc):=\frac{\bc}{Lr}.
\end{equation}
We first observe that, if \begin{equation}\label{eq:tu1}
	|t|\leqslant C_1 |\bu_r|,
\end{equation} for an explicit uniform constant $C_1>0$, then uniformly for $\bx\in\operatorname{Supp}(w_1)$ (in particular $\nabla G(\bx)\gg 1$), the first derivative of the oscillatory term satisfies
$$\left|t\nabla G(\bx)-\bu_r\right|\gg |\bu_r|,$$ and the second partial derivative (determined by the Hessian matrix) has norm $\asymp |t|\ll |\bu_r|$. So by \cite[Lemma 10]{H-Bdelta}, under \eqref{eq:tu1} rapid decay applies: $$\int_{\BR^3}w_1(\bx)\e\left(t(F(\bx)-m_0)-\frac{\bb}{r}\cdot\bx\right)\operatorname{d}\bx\ll_M |\bu_r|^{-M}$$ This is analogous to the case where we apply the ``easy estimate''. We thus have that the $q$-sum contributes 
\begin{align*}
	&\ll_M \sum_{q\ll Q(\log Q)^{-\varepsilon_0}} \frac{\left|\widetilde{S}_{q}(\bc)\right|}{q^3}|\bu_{\frac{q}{Q}}|^{-M}\int_{\BR}\left|P_{\frac{q}{Q}}(t)\right|\operatorname{d}t\\ &\ll |\bc|^{-M} (\log Q)^{-\varepsilon_0(M-1)}\sum_{q\ll Q} \frac{\left|\widetilde{S}_{q}(\bc)\right|}{q^3}\ll_\varepsilon |\bc|^{2-M+\varepsilon} (\log Q)^{-\varepsilon_0(M-1)+1}
\end{align*}
by Lemmas \ref{le:Kloosterman2} (with partial summation) and \ref{le:Prt}. 

We henceforth proceed under the assumption
\begin{equation}\label{eq:tu2}
		|t|>C_1 |\bu_r|.
\end{equation} The dissection process in \cite[p. 185--186]{H-Bdelta} shows that, we can express for any $\bx\in\operatorname{Supp}(w_1)$, $$w_1(\bx)=\frac{1}{\delta_0^3}\int_{\BR}w_{\delta_0}\left(\frac{\bx-\by}{\delta_0},\by\right)\operatorname{d}\by,$$ where $\delta_0>0$ and $w_{\delta_0}$ is constructed in \cite[Lemma 2]{H-Bdelta} depending on $\delta_0$. We exchange $\bx$ to $\bz$ bearing the relation 
\begin{equation}\label{eq:relationxyz}
\bx=\by+\delta_0\bz.
\end{equation}
We choose $$\delta_0=\frac{C_2}{|t|^{\frac{1}{2}}}$$ for a sufficiently small but uniform constant $C_2>0$, so that for every $\by$ which is within distance $\delta_0$ of a point $\bx\in\operatorname{Supp}(w)$, one has $|\nabla G(\by)|\gg 1$ thanks to \eqref{eq:condkey} \eqref{eq:tu2}. Note that our choice of $\delta_0$ is different from \cite{H-Bdelta} but essentially equivalent. The advantage here is that $\delta_0$ is independent of $\bu_r$ (hence independent of $r$).

For every fixed $r>0$, we let \begin{equation}\label{eq:R}
	R_r(\bc):=|\bu_r|^{\varepsilon_1},
\end{equation} for a fixed $0<\varepsilon_1<\frac{1}{3}$. In particular, we have $R_r(\bc)^3\leqslant |\bu_r|$, and the condition $r\leqslant (\log Q)^{-\varepsilon_0}$ resulting from \eqref{eq:condkey} guarantees that	$R_r(\bc)\gg (\log Q)^{\varepsilon_0\varepsilon_1}$.
We call a pair $(\by,t)$ \emph{good} (with respect to $\bc,r$) if \begin{equation}\label{eq:goodpair}
	|t\nabla F(\by)-\bu_r|\geqslant R_r(\bc)|t|^{\frac{1}{2}}.
\end{equation} As in \cite[p. 186]{H-Bdelta}, for every such fixed good pair $(\by,t)$, the first derivative of the function $tG(\by+\delta_0\bz)-\bu_r\cdot(\by+\delta_0\bz)$ with respect to $\bz$ is $\gg R_r(\bc)$, while the second derivative is $O(1)$, hence the first derivative test \cite[Lemma 10]{H-Bdelta} again applies  and yields that for any fixed good pair $(\by,t)$, \begin{equation}\label{eq:goodpaircons}
\int_{\BR^3} w_{\delta_0}(\bz,\by)\e\left(tG(\bx)-\bu_r\cdot\bx\right)\operatorname{d}\bz\ll_M R_r(\bc)^{-M}.
\end{equation}
Going back to the $q$-sum to be dealt with, since $|\by|=O(1)$, thanks to the rapid decay \eqref{eq:goodpaircons}, the contribution from ``good pairs'' (with respect to $\bc$ and all $r=\frac{q}{Q}\in I_1/Q$) is bounded by
\begin{align*}
	&\ll_M \sum_{q\ll Q(\log Q)^{-\varepsilon_0}} \frac{\left|\widetilde{S}_{q}(\bc)\right|}{q^3}R_{\frac{q}{Q}}(\bc)^{-M}\int_{\BR}\left|P_{\frac{q}{Q}}(t)\right|\operatorname{d}t\\ &\ll |\bc|^{-\varepsilon_1 M}(\log Q)^{-\varepsilon_0( M-1)}\sum_{q\ll Q} \frac{\left|\widetilde{S}_{q}(\bc)\right|}{q^3}\\ &\ll_{\varepsilon}|\bc|^{2-\varepsilon_1 M+\varepsilon}(\log Q)^{-\varepsilon_0 (M-1)+1},
\end{align*} by Lemmas \ref{le:Kloosterman2} and \ref{le:Prt}. We take $M$ large enough (depending on $\varepsilon_0,\varepsilon_1$) to ensure that the exponents appearing above are sufficiently negative.  

We now turn to analysing the \emph{bad} pairs $(\by,t)$ (with respect to $\bc,r$) satisfying 
\begin{equation}\label{eq:badpair}
	|t\nabla G(\by)-\bu_r|<R_r(\bc)|t|^{\frac{1}{2}}.
\end{equation}as opposed to \eqref{eq:goodpair}.
Following the discussion in \cite[p. 186--187]{H-Bdelta}, since $|\nabla G(\by)|\asymp 1$ and $R_r(\bc)\ll |\bu_r|^{\frac{1}{3}}\ll t^{\frac{1}{3}}$ by \eqref{eq:tu2}, we must have 
\begin{equation}\label{eq:bad1}
	|t|\asymp |\bu_r|.
\end{equation}
Moreover, for any fixed $\bc,r,t$, by \cite[Lemma 21]{H-Bdelta}, the area of $\by$ satisfying \eqref{eq:badpair} is \begin{equation}\label{eq:badarea}
 \int_{\substack{\by\in\BR^3\\ \eqref{eq:badpair} \text{ holds}}}\ll  |t|^{-\frac{3}{2}}R_r(\bc)^3.
\end{equation}  
Turning to the $q$-sum, for any fixed $\bc,\by,t$, let $\mathfrak{I}=\mathfrak{I}(\bc;\by;t)$ be the interval of $q$ determined by the conditions \eqref{eq:condkey} \eqref{eq:tu2} \eqref{eq:badpair}, whenever non-empty. Moreover for fixed $\bz$, we let (recall the relation \eqref{eq:relationxyz}) $$\CG(\bx,t;\bc):=\sum_{q\in \mathfrak{I}}\frac{\e_{qL^2}(\bc\cdot\blambda_N)\widetilde{S}_{q}(\bc)\e_{\frac{q}{Q}}\left(-\frac{\bc}{L}\cdot\bx\right)}{q^2}\times\frac{P_{\frac{q}{Q}}(t)}{q}.$$ 
Then the $q$-sum within the range \eqref{eq:condkey} against ``bad pairs'' is expressed as
\begin{align}
	&\sum_{q\in I_1}\frac{\e_{qL^2}(\bc\cdot\blambda_N)\widetilde{S}_{q}(\bc)}{q^3}\iint_{\substack{(\by,t) \text{ bad}\\ \eqref{eq:tu2}\text{ holds}}}P_{\frac{q}{Q}}(t)\operatorname{d}\by\operatorname{d}t\int_{\bz\in\BR^3}w_{\delta_0}(\bz,\by)e\left(tG(\bx)-\bu_r\cdot\bx\right)\operatorname{d}\bz\nonumber\\ =& \iiint_{(t,\by,\bz)\in\BR\times\BR^3\times\BR^3}w_{\delta_0}(\bz,\by)e\left(tG(\bx)\right)\CG(\bx,t;\bc)\operatorname{d}t\operatorname{d}\by\operatorname{d}\bz. \label{eq:intbadpair}
\end{align}
Then on recalling \eqref{eq:CFtc}, partial summation gives
\begin{equation}\label{eq:CGpartial}
	\CG(\bx,t;\bc)=-\frac{1}{Q}\int_{\mathfrak{I}/Q}\CF_{\bx,\bc}(rQ)\frac{\partial}{\partial r}\left(\frac{P_{r}(t)}{r}\right)\operatorname{d}r +O\left(\sup_{s\in\mathfrak{I}}\left|\CF_{\bx,\bc}(s)\frac{P_{\frac{s}{Q}}(t)}{s}\right|\right).
\end{equation}
Bringing \eqref{eq:CGpartial}  back to \eqref{eq:intbadpair} and exchanging the order of integration back, on using \eqref{eq:badarea}, Proposition \ref{prop:sumSalie} and Lemma \ref{le:Prt}, the overall contribution of the integral is

 \begin{align*}
 	&= -\frac{1}{Q}\int_{r\in I_1/Q}\operatorname{d}r\iint_{\substack{(\by,t) \text{ bad}\\ \eqref{eq:tu2}\text{ holds}}}\frac{\partial}{\partial r}\left(\frac{P_{r}(t)}{r}\right)\operatorname{d}\by\operatorname{d}t\int_{\bz\in\BR^3}\CF_{\bx,\bc}(rQ)w_{\delta_0}(\bz,\by)e\left(tG(\bx)\right)\operatorname{d}\bz\\
 	&\ll_{\varepsilon} \frac{|\bc|^{2+\varepsilon}}{(\log Q)^{\frac{1}{4}(1-\frac{\sqrt{2}}{2})-\varepsilon}}\int_{r\in I_1/Q}r\operatorname{d}r \int_{\substack{t\in\BR\\ \eqref{eq:bad1}\text{ holds}}}\left|\frac{\partial}{\partial r}\left(\frac{P_{r}(t)}{r}\right)\right| \operatorname{d}t\int_{\substack{\by\in\BR^3\\ \eqref{eq:badpair} \text{ holds}}}\operatorname{d}\by\\ &\ll  \frac{|\bc|^{2+\varepsilon}}{(\log Q)^{\frac{1}{4}(1-\frac{\sqrt{2}}{2})-\varepsilon}}\int_{r\in I_1/Q}rR_r(\bc)^3 \operatorname{d}r \int_{\substack{t\in\BR\\ \eqref{eq:bad1}\text{ holds}}}\left|\frac{\partial}{\partial r}\left(\frac{P_{r}(t)}{r}\right)\right|t^{-\frac{3}{2}} \operatorname{d}t\\ &\ll \frac{|\bc|^{2+\varepsilon}}{(\log Q)^{\frac{1}{4}(1-\frac{\sqrt{2}}{2})-\varepsilon}}\int_{0}^{\max\left(|\bc|^{M_0},(\log Q)^{\varepsilon_0}\right)^{-1}}rR_r(\bc)^3\times r^{-2}|\bu_r|^{-\frac{1}{2}}\operatorname{d}r\\ &\ll \frac{|\bc|^{\frac{3}{2}+3\varepsilon_1+\varepsilon}}{(\log Q)^{\frac{1}{4}(1-\frac{\sqrt{2}}{2})-\varepsilon}}\int_{0}^{|\bc|^{-M_0}}\frac{1}{r^{\frac{1}{2}+3\varepsilon_1}}\operatorname{d}r\\ &\ll |\bc|^{\frac{3}{2}+3\varepsilon_1-M_0(\frac{1}{2}-3\varepsilon_1)+\varepsilon}(\log Q)^{-\frac{1}{4}(1-\frac{\sqrt{2}}{2})+\varepsilon}.
 \end{align*}
 
 We now estimate the overall contribution of the variation term in \eqref{eq:CGpartial}. For each individual $t\in\BR$ and $r_0\in I_1/Q$, on using Proposition \ref{prop:sumSalie}, we have 
 $$\CF_{\bx,\bc}(r_0Q)\frac{P_{r_0}(t)}{r_0Q}\ll_\varepsilon \frac{|\bc|^{2+\varepsilon}|P_{r_0}(t)|}{(\log Q)^{\frac{1}{4}(1-\frac{\sqrt{2}}{2})-\varepsilon}}.$$
Our argument above, on using again Lemma  \ref{le:Prt}, shows that the overall contribution of bad pairs for such an $r_0$ is
\begin{align*}
	&\ll_{\varepsilon} \frac{|\bc|^{2+\varepsilon}}{(\log Q)^{\frac{1}{4}(1-\frac{\sqrt{2}}{2})-\varepsilon}}\int_{\substack{t\in\BR\\ \eqref{eq:bad1}\text{ holds}}} \left|P_{r_0}(t)\right|\operatorname{d}t\int_{\substack{\by\in\BR^3\\ \eqref{eq:badpair} \text{ holds}}}\operatorname{d}\by\\ &\ll |\bc|^{\frac{3}{2}+3\varepsilon_1+\varepsilon} r_0^{\frac{1}{2}-3\varepsilon_1} (\log Q)^{-\frac{1}{4}(1-\frac{\sqrt{2}}{2})+\varepsilon}.
\end{align*}
This bound being an ascending function on $r_0$, on taking $r_0=\max\left(|\bc|^{M_0},(\log Q)^{\varepsilon_0}\right)^{-1}$, we get the same upper bound as above. 

Finally it remains to carefully adjust the choice of $M_0,\varepsilon_0,\varepsilon_1$, thereby completing the proof.
\end{proof}
 				
 	\begin{remark}\label{rmk:ordc}
 				We now explain the crux of the issue when summing over relatively small $q$ (i.e. within the range \eqref{eq:condkey}) in the proof of Theorem \ref{thm:cord}. Following the argument in \cite[p. 182--183]{H-Bdelta}, (as explained before Lemma \ref{le:easyhardpartial}) the derivative of the oscillatory factor $\e_{r}(\star)$ produces $r^{-2}$ and due to the inhomogeneous equation $F(\bx)-m_0$ it is unclear how to eliminate its appearance, contrary to the homogeneous case. If one runs partial summation in the naive way, this leads to divergence of various integrals against the partial derivative of oscillatory integrals with respect to $r$ around $0$. Our idea to resolve this problem is, whenever rapid decay furnished by the first derivative test of the oscillatory integral is not available, where the ``harder estimate'' is brought into play, we run partial summation in a way that encapsulates the oscillatory factor by means of Proposition \ref{prop:sumSalie}, while putting the partial derivative directly into (the Fourier transform of) the function $h$. 
 			\end{remark}

 		\subsection{Proof of Theorem \ref{thm:cneq0sum}}
 		Again using the ``easy estimate'' in Lemma \ref{le:easyhardest}, we can restrict the $\bc$-sum to those $|\bc|\ll_{\varepsilon} Q^\varepsilon$ with an arbitrarily high power-saving error term. 
 		
 		For ordinary $\bc$'s, by Theorem \ref{thm:cord}, upon choosing the integer $M>0$ large enough, the subsequent $\bc$-sum is clearly convergent, and hence the contribution of these $\bc$ is $\ll \frac{N^\frac{3}{2}}{(\log N)^{\frac{1}{4}(1-\frac{\sqrt{2}}{2})-A_0}}$. 
 		
 		It remains to sum over the exceptional $\bc$'s. In view of Theorem \ref{thm:exec}, we define 
 		$$\CK_h:=\frac{1}{L^6}\sum_{\substack{\bc\in\BZ^3\setminus\{\boldsymbol{0}\}\\\text{exceptional}}}\widetilde{\eta}_h(\bc),$$ which according to \eqref{eq:etacbd} is clearly convergent, and the $\bc$-sum with the restriction $|\bc|\ll_{\varepsilon} Q^\varepsilon$ approximates $\CK_h(\bc)$ also with an arbitrarily high power-saving error term. Summing over them adds to the error term $O_\varepsilon(Q^{-\lambda_0+\varepsilon})$ which becomes $O_\varepsilon(N^{\frac{3}{2}-\lambda_0+\varepsilon})$ in \eqref{eq:execstep} and is negligible compared to the log-saving. \qed
 		
 \subsection{The case $\bc=\boldsymbol{0}$}\label{se:c=0}
We shall only give the detailed proof assuming \eqref{eq:notsq} and sketch the proof for the case \eqref{eq:m0deltasq} which is analogous to \cite{Huang}.
 \begin{proposition}\label{prop:singnonsq}
 	 Assume \eqref{eq:notsq}. We then have
 	 $$\sum_{q\leqslant X}\frac{\widetilde{S}_{q}(\boldsymbol{0})}{q^3}=L^4\BL(1,\psi_0)\widetilde{\mathfrak{S}}_F(p_0;L,\blambda)+O_\varepsilon(X^{-\frac{1}{2}+\varepsilon}),$$
 	 $$\sum_{q\leqslant X}\widetilde{S}_{q}(\boldsymbol{0})\ll_{\varepsilon} X^{\frac{5}{2}+\varepsilon}.$$
 \end{proposition}
\begin{proof}
 By  the Chinese remainder theorem, $\widetilde{S}_{q}(\boldsymbol{0})$ is multiplicative in $q$.
 By Lemma \ref{le:S1}, for $q\in\BN$ with $(q,m_0p_0\Omega)=1$, we can compute
 $$\widetilde{S}^{(1)}_{q}(\boldsymbol{0})=\iota_q^3\left(\frac{\Delta_F}{q}\right)\sum_{\substack{a\bmod q\\ (a,q)=1}}\left(\frac{a}{q}\right)\e_{q}(-am_0p_0^{2h})=q^2\mu^2(q)\psi_0(q).$$
 Therefore, the formal series
 $\varPi(s):=\sum_{q=1}^\infty \frac{\widetilde{S}_{q}(\boldsymbol{0})}{q^s}$ factorises as $$\varPi(s)=\BL(s-2,\psi_0)\widetilde{\nu}(s),$$ where by Lemma \ref{leS2bd},
 \begin{align*}
 	\widetilde{\nu}(s)&:=\prod_{p<\infty}\left(1-\frac{\psi_0(p)}{p^{s-2}}\right)\sum_{t=0}^{\infty}\frac{\widetilde{S}_{p^t}(\boldsymbol{0})}{p^{ts}}\\ &=\prod_{p\nmid m_0p_0\Omega}\left(1-\frac{1}{p^{2(s-2)}}\right)\times\prod_{p\mid m_0p_0\Omega}\left(1-\frac{\psi_0(p)}{p^{s-2}}\right)\sum_{t=0}^{\infty}\frac{\widetilde{S}_{p^t}(\boldsymbol{0})}{p^{ts}}
 \end{align*}  is absolutely convergent for $\Re(s)>\frac{5}{2}$.

We now compute explicitly the $p$-adic densities above. 
First it is easy to see that when $p\neq p_0$, we have, on writing $m_p:=\operatorname{ord}_{p}(L)$,
\begin{align*}
	&\sum_{t=0}^{\infty}\frac{\widetilde{S}_{p^t}(\boldsymbol{0})}{p^{3t}}\\=&\lim_{k\to\infty}\frac{\#\{\bx\in(\BZ/p^{k+2m_p}\BZ)^3:F(\bx)\equiv m_0p_0^{2h}\bmod p^{k+2m_p},\bx\equiv p_0^{h}\blambda\bmod p^{m_p}\}}{p^{2k}}\\ =&p^{4m_p}\sigma_{p}(\CV_1;(L,\blambda)),
\end{align*}
by the non-singular change of variable $\bx\mapsto p_0^{h}\bx$.

We now turn to analysing the $p_0$-adic factor. By Hensel's lemma (see also \cite[(20.128) p. 479]{Iwaniec-Kolwalski}), there exists $D\in\BN$ depending only on the quadratic form $F$ (only need $F\bmod D$ is geometrically integral) and $L$ such that $L\mid D$ and for every $E\in\BN$ with $D\mid E$ and for every $M\in\BZ$,  we have
\begin{align*}
	&\#\{\bx\bmod E:F(\bx)\equiv M\bmod E,\bx\equiv \blambda\bmod L\}\\ =&\left(\frac{E}{D}\right)^2\#\{\bx'\bmod D:F(\bx')\equiv M\bmod D,\bx'\equiv \blambda\bmod L\}.
\end{align*}
In particular, extracting the $p_0$-adic part by the Chinese remainder theorem,  this shows that, whenever $2h>k_0:=\operatorname{ord}_{p_0}(D)$ \begin{align*}
	\sum_{t=0}^{\infty}\frac{\widetilde{S}_{p^t}(\boldsymbol{0})}{p^{3t}}&=\lim_{k\to\infty}\frac{\#\{\bx\in(\BZ/p_0^k\BZ)^3:F(\bx)\equiv m_0p_0^{2h}\bmod p_0^k\}}{p_0^{2k}}\\ &=\frac{\#\{\bx\in(\BZ/p_0^{k_0}\BZ)^3:F(\bx)\equiv m_0p_0^{2h}\bmod p_0^{k_0}\}}{p_0^{2k_0}}=\sigma_{p_0}(\CV_0).
\end{align*}
Therefore  we have $$\widetilde{\nu}(3)=L^4\widetilde{\mathfrak{S}}_F(p_0;L,\blambda).$$
The asymptotics above follow from Perron's formula as in \cite[Proof of Lemmas 30 \& 31]{H-Bdelta}. (Because in the first case $\varPi(3)$ is the residue of the simply pole $\varPi(s+3)X^s/s$ at $s=0$, and in the second case $\varPi(s)X^s/s$ has no pole for $\frac{5}{2}<\Re s$ and hence has residue zero.) \end{proof}

\begin{proof}[Proof of Proposition \ref{prop:c=0} assuming \eqref{eq:notsq}]
	We fix $\varepsilon_0>0$ and write $A_t=\sum_{q\leqslant t}\widetilde{S}_{q}(\boldsymbol{0})$. Then by partial summation,
	\begin{align*}
		&\sum_{Q^{1-\varepsilon_0}<q\ll Q}\frac{\widetilde{S}_{q}(\boldsymbol{0})\widetilde{\CI}_{\frac{q}{Q}}(w;\boldsymbol{0})}{q^3}\\&=-\int_{Q^{1-\varepsilon_0}}^{cQ}A_t\frac{\partial}{\partial t}\left(\frac{\widetilde{\CI}_{\frac{t}{Q}}(w;\boldsymbol{0})}{t^3}\right)\operatorname{d}t+\left[A_t\frac{\widetilde{\CI}_{\frac{t}{Q}}(w;\boldsymbol{0})}{t^3}\right]_{Q^{1-\varepsilon_0}}^{cQ}\\ &\ll_\varepsilon Q^{-\frac{1}{2}+\varepsilon}\int_{Q^{-\varepsilon_0}}^{c} r^{\frac{5}{2}+\varepsilon}\left(\left|\frac{1}{r^3}\frac{\partial\widetilde{\CI}_{r}(w;\boldsymbol{0})}{\partial r}\right|+\left|\frac{\widetilde{\CI}_{r}(w;\boldsymbol{0})}{r^4}\right|\right) \operatorname{d}r+Q^{-\frac{1}{2}(1-\varepsilon_0)+\varepsilon}\\ & \ll_\varepsilon Q^{-\frac{1}{2}+\varepsilon}\int_{Q^{-\varepsilon_0}}^{c} r^{-\frac{3}{2}+\varepsilon}\operatorname{d}r+Q^{-\frac{1}{2}(1-\varepsilon_0)+\varepsilon}\ll Q^{-\frac{1}{2}(1-\varepsilon_0)+\varepsilon}.
	\end{align*}
On the other hand, using \cite[Lemma 13]{H-Bdelta}, 
\begin{align*}
	&\sum_{q\ll Q^{1-\varepsilon_0}}\frac{\widetilde{S}_{q}(\boldsymbol{0})\widetilde{\CI}_{\frac{q}{Q}}(w;\boldsymbol{0})}{q^3}\\ =&\sum_{q\ll Q^{1-\varepsilon_0}}\frac{\widetilde{S}_{q}(\boldsymbol{0})}{q^3}\widetilde{\CI}(w)+O_M\left(\sum_{q\ll Q^{1-\varepsilon_0}}\frac{|\widetilde{S}_{q}(\boldsymbol{0})|}{q^3}\left(\frac{q}{Q}\right)^M\right)\\ =&L^4\BL(1,\psi_0)\widetilde{\mathfrak{S}}_F(p_0;L,\blambda)\widetilde{\CI}(w)+O_{\varepsilon,M}(Q^{-\frac{1}{2}(1-\varepsilon_0)+\varepsilon}+Q^{-M\varepsilon_0+\varepsilon}).
\end{align*}
It remains to adjust the choice of $M$ and $\varepsilon_0$.
\end{proof}

\begin{proof}[Sketch of proof of  Proposition \ref{prop:c=0} assuming \eqref{eq:m0deltasq}]
On reusing the argument in the proof of Proposition \ref{prop:singnonsq} above, combined with the proof of \cite[Proposition 5.9]{Huang} we can show that
	$$\sum_{q\leqslant X}\widetilde{S}_{q}(\boldsymbol{0})=\frac{1}{3}L^4\widetilde{\mathfrak{S}}_F(p_0;L,\blambda)X^3+O_\varepsilon(X^{\frac{5}{2}+\varepsilon}),$$
	$$\sum_{q\leqslant X}\frac{\widetilde{S}_{q}(\boldsymbol{0})}{q^3}=\widetilde{\mathfrak{S}}_F(p_0;L,\blambda)\left(L^4 \log X+\gamma\right)+O_\varepsilon(X^{-\frac{1}{2}+\varepsilon}),$$ where $\gamma$ is the Euler constant. 
	Using these, the proof of \cite[Theorem 5.8]{Huang} goes \emph{mutatis mutandis}.
\end{proof}

\subsection{Completion of proof of main theorems}\label{se:completeproof}
We recall \eqref{eq:Poisson} and the choice of $Q$ \eqref{eq:Q}. Taking Theorem \ref{thm:cneq0sum} and Proposition \ref{prop:c=0} into account, we let $$a_h:=a_0+L^2\CK_h. $$ assuming \eqref{eq:m0deltasq}, while we let $$b_h:= L^2\CK_h$$ assuming \eqref{eq:notsq}. We complete the proof by letting $A:=\frac{1}{4}(1-\frac{\sqrt{2}}{2})-A_0$. \qed

		\section*{Acknowledgements}
			We are very grateful to Dasheng Wei and Fei Xu for many helpful discussions. We thank Tim Browning and Roger Heath-Brown for their interests and comments.

\end{document}